\documentclass[12pt,reqno]{amsart}

\usepackage[a4paper]{geometry}
\usepackage[T2A]{fontenc}
\usepackage{amssymb,amsmath,amsthm}
\usepackage[english]{babel}
\usepackage{xcolor}

\newtheorem{theorem}{Theorem}
\newtheorem{proposition}{Proposition}
\newtheorem{corollary}{Corollary}

\theoremstyle{definition}

\newtheorem{definition}{Definition}
\newtheorem{example}{Example}
\newtheorem{remark}{Remark}

\def\Res{\mathop{\fam 0 Res}}
\def\oo#1{\mathbin{ {}_{(#1)}}}
\def\soo#1{\mathbin{ {*}_{(#1)}}}

\title[On the locality of formal distributions]{On the locality of formal distributions over pre-Lie and Novikov algebras}

\thanks{The work was supported by the Program of Fundamental Research RAS (project FWNF-2022-0002).}

\author{L.~A. Bokut, P.~S. Kolesnikov}

\address{Sobolev Institute of Mathematics, Novosibirsk, Russia}

\begin{document}

\begin{abstract}
The Dong Lemma in the theory of vertex algebras states that the locality
property of formal distributions over a Lie algebra is preserved
under the action of a vertex operator. A~similar statement is known for
associative algebras. We study local formal distributions over pre-Lie
(right-symmetric), pre-associative (dendriform), and Novikov algebras
to show that the analogue of the Dong Lemma holds for Novikov algebras
but does not hold for pre-Lie and pre-associative ones.
\end{abstract}

\maketitle

\section*{Introduction}

Conformal algebras were introduced in \cite{Kac1996}
as a tool for the theory of vertex algebras in
mathematical physics and reprsentation theory.
A keystone in the construction of a vertex algebra is the Dong Lemma stating that,
given a system of pairwise mutually local formal distributions over a Lie algebra,
all conformal $n$-products of these distributions are also
mutually local.
The Dong Lemma is also essential for constructing free (associative and Lie) conformal algebras
since it is
enough to fix locality on the generators to define a universal structure \cite{Roit99}.

Associative and Lie conformal algebras, as well as Jordan conformal algebras
are well studied in a series of papers (see, e.g.,
\cite{DK1998, FK2002, KacCant_JSuper, Kol2006Adv, Zelm2000}).
From the categorial point of view,
if a class $\mathfrak M$ of ``ordinary'' algebras over a field $\Bbbk $
is presented by
morphisms of operads
$\mathcal O_{\mathfrak M}\to \mathrm{Vec}_\Bbbk $,
where $\mathcal O_{\mathfrak M}$ is the operad corresponding to the class $\mathfrak{M}$
and $\mathrm{Vec}_\Bbbk $ is the multi-category of linear spaces over $\Bbbk $, then
$\mathfrak M$-conformal algebras are presented by
morphisms $\mathcal O_{\mathfrak M}\to \Bbbk [\partial]\text{-mod}$,
where $\Bbbk [\partial]\text{-mod}$ is the multi-category (pseudo-tensor category)
of modules over the polynomial Hopf algebra (see \cite{BDK2001} for details).

In this way, other varieties of conformal algebras appear naturally, and the study
of their structure and relations between different varieties is an interesting
algebraic problem. The class of Poisson conformal algebras was studied in \cite{KolIJAC2021},
right-symmetric (pre-Lie) conformal algebras were introduced in \cite{HongLi2015}
along with Novikov conformal algebras.
The latter class turns to be closely related with Novikov--Poisson algebras,
and the study of Novikov conformal algebras helps solving problems
in this area \cite{KolNest2023}.
Dendriform (pre-associative) conformal algebras appeared in \cite{HongBai2021, Yuan2022},
they are related to deformation theory problems for conformal algebras.

Recall that the class of right-symmetric algebras consists of all pairs $(V,\circ )$,
where $V$ is a linear space and $\circ $ is a bilinear operation on $V$,
$(x,y)\mapsto x\circ y$,
such that
\begin{equation}\label{eq:RSym}
(x\circ y)\circ z -  x\circ (y\circ z )
 = (x\circ z)\circ y -  x\circ (z\circ y )
\end{equation}
for all $x,y,z\in V$.
If, in addition, the identity
\begin{equation}\label{eq:LCom}
 x\circ (y\circ z) = y\circ (x\circ z)
\end{equation}
holds on $V$ then $(V,\circ )$ is said to be a Novikov algebra.
This notion emerged in \cite{GD1979} and \cite{BalNov},
where the identities \eqref{eq:RSym}, \eqref{eq:LCom} (or their opposite versions)
were used to describe relations on the coordinates of a rank 3 tensor
appeared in different studies in functional analysis and partial differential equations.

For example, if $A$ is a commutative (and associative) algebra with a derivation $d:A\to A$
then the same space $A$ equipped with the new operation $(x,y)\mapsto x\circ y = d(x)y$
is a Novikov algebra denoted $A^{(d)}$.
It was shown in \cite{DzhLofwall2002} and \cite{BCZ2017} that
for every Novikov algebra $V$ there exists a commutative algebra $A$
with a derivation $d$
such that $V$ is isomorphic to a subalgebra of $A^{(d)}$.
A similar statement was derived in \cite{KolNest2023} for quadratic Novikov conformal algebras,
those related with Novikov--Poisson algebras.
In order to study this embedding problem more precisely and generally,
the notion of a free Novikov conformal algebra is needed. To that end, an analogue
of the Dong Lemma for formal distributions over Novikov algebras is crucially important.

The purpose of this paper is to establish an analogue of the Dong Lemma
for Novikov algebras and show that this statement does not hold
for pre-Lie and pre-associative algebras.

Throughout the paper,
$\Bbbk $ is a field of characteristic zero and $\mathbb Z_+$ stands for the set of
non-negative integers.

\section{Formal distributions, locality, and the OPE formula}

Let $A$ be an algebra, i.e., a linear space over a field $\Bbbk $,
$\mathrm{char}\,\Bbbk =0$,
equipped with a bilinear operation (multiplication)
$\mu : A\times A \to A$, $\mu(a,b)=ab$,
for $a,b\in A$.
We do not assume the operation $\mu $ is associative or commutative.

A formal distribution over $A$ is a two-side infinite formal power series
\begin{equation}\label{eq:GenericDistr}
 a(z) = \sum\limits_{n\in \mathbb Z} a_n z^{-n-1}, \quad a_n\in A,
\end{equation}
where $z$ is a formal variable. The space of all formal distributions over $A$
is denoted $A[[z,z^{-1}]]$. In contrast to ordinary power series $A[[z]]$
or Lawrent series over $A$, it is in general impossible to multiply
formal distributions due to potentially infinite sums emerging at coefficients.
However, given two formal distributions $a(z), b(z)\in A[[z,z^{-1}]]$ as above,
one may easily define the product
\[
 a(w)b(z) = \sum\limits_{n,m\in \mathbb Z} a_nb_m w^{-n-1}z^{-m-1} \in A[[w,w^{-1},z,z^{-1}]].
\]

A pair $(a(z),b(z))$ of formal distributions over an algebra $A$ is said to be {\em local}
if there exists $N\in \mathbb Z_+$ such that
\begin{equation}\label{eq:LocalityForDist}
 a(w)b(z) (w-z)^N = 0.
\end{equation}
The minimal such $N$ is said to be the {\em locality value} of the pair $(a(z), b(z))$.

Two formal distributions $a(z), b(z)\in A[[z,z^{-1}]]$ are said to be mutually local
if $(a(z),b(z))$ and $(b(z),a(z))$ are local pairs.

\begin{remark}
Let $\partial $ be the formal derivation
of formal distributions, i.e.,
if $a(z)$ is given by \eqref{eq:GenericDistr} then
\[
 (\partial a)(z) = \sum\limits_{n\in \mathbb Z}(-n-1) a_n z^{-n-2} =
 -n\sum\limits_{n\in \mathbb Z} a_{n-1} z^{-n-1}.
\]
If $a(w)b(z)(w-z)^N=0$ then $(\partial a)(w)b(z) (w-z)^{N+1} = 0$
as well as
$a(w)(\partial b)(z) (w-z)^{N+1} = 0$.
\end{remark}

\begin{example}\label{exmp:Weyl}
Let $A$ be the associative algebra
generated by $q$, $t$, $t^{-1}$
such that $qt-tq=1$. (This is the algebra of differential
operators on the algebra of Lawrent polynomials $\Bbbk [t,t^{-1}]$.)
Then the formal distributions
\[
 q^{(m)}(z) = \sum\limits_{n\in \mathbb Z} \dfrac{1}{m!} t^nq^m z^{-n-1} \in A[[z,z^{-1}]],
 \quad m\in \mathbb Z_+,
\]
are pairwise mutually local and the locality value for $q^{(m)}(z), q^{(k)}(z)$
is equal to
$N = m+1$.
\end{example}

\begin{example}\label{exmp:NovikovLie}
Let $V$ be a Novikov algebra with an operation
$(a,b)\mapsto a\circ b$, $a,b\in V$.
Then the space of Lawrent polynomials $V[t,t^{-1}]$
equipped with a new operation
\[
 [at^n,bt^m] = (n(b\circ a)-m(a\circ b))t^{m+n-1}, \quad a,b\in V,\ n,m\in \mathbb Z,
\]
is a Lie algebra, and the series
of the form
\[
 a(z) = \sum\limits_{n\in \mathbb Z} at^n z^{-n-1}, \quad a\in V,
\]
are pairwise local with $N\le 2$.
\end{example}

In particular, for 1-dimensional commutative (hence, Novikov)
algebra $V=\Bbbk $ such a series generates the Virasoro conformal Lie algebra
\cite{DK1998}.

It follows by direct computations
\cite{KacForDist} that
a pair of formal distributions $(a(z),b(z))$ over an algebra $A$
is local if and only if
there exists $N\in \mathbb Z_+$ such that
\begin{equation}\label{eq:LocalityCoeff}
 \sum\limits_{s\in \mathbb Z_+} (-1)^s \binom{N}{s} a_{n-s}b_{m+s}  = 0
\end{equation}
for all $n,m\in \mathbb Z$.

Given a local pair of formal distributions $a(z),b(z)\in A[[z,z^{-1}]]$,
their {\em $\lambda $-product}
is a polynomial in $\lambda $ with coefficients
in $A[[z,z^{-1}]] $ defined as follows:
\[
 (a(z)\oo\lambda b(z) ) = \sum\limits_{n\in \mathbb Z_+} \dfrac{\lambda^n}{n!} c^{(n)} (z),
\]
where
\begin{equation}\label{eq:n-prod}
 c^{(n)}(z) = \Res\limits_{w=0} a(w)b(z)(w-z)^n, \quad n\in \mathbb Z_+.
\end{equation}
Here $\Res\limits_{w=0}$ denotes the residue of a formal distribution, i.e.,
the coefficient at $w^{-1}$.
The coefficients $c^{(n)}(z)$ are called {\em $n$-products}
of $a(z)$ and $b(z)$, denoted
$(a\oo{n} b)(z)$.
It is easy to compute (see also \cite{KacForDist, Roit99}) that
\[
 c^{(n)}(z) = \sum\limits_{m\in \mathbb Z} (a\oo{n} b)_m z^{-m-1},
 \quad
 (a\oo{n} b)_m = \sum\limits_{s\in \mathbb Z_+} (-1)^s \binom{n}{s} a_{n-s}b_{m+s}.
\]

The importance of $n$-products becomes clear from the following observation.
Suppose we are given a local pair $(a(z),b(z))$  of formal distributions.
It follows from \eqref{eq:LocalityCoeff} that every element in $A$ of the form
$a_kb_m$, $k,m\in \mathbb Z$, may be expressed as a linear combination
of
\[
 a_0b_{k+m}, a_1b_{k+m-1}, \dots, a_{N-1}b_{k+m-N+1}.
\]
The explicit form of such an expression is known as the operator product
expansion (OPE) formula: 
\[
 a(w)b(z) = \sum\limits_{s=0}^{N-1} \dfrac{1}{s!} (a\oo s b)(z)
  \dfrac{\partial^s}{\partial z^s} \delta(w-z),
\]
i.e., the product of two distributions in different variables
expands into a finite sum with respect to derivatives
of the formal delta-function
\[
\delta(w-z) = \sum\limits_{r+s=-1}w^rz^s \in \Bbbk [[w,w^{-1},z,z^{-1}]].
\]

\begin{theorem}[Dong Lemma, see, e.g., \cite{FrenBZvi}]\label{thm:DongLemma}
 Let $A$ be a Lie algebra, and let
$a(z),b(z),c(z) \in A[[z,z^{-1}]]$
be pairwise mutually local distributions.
Then for every $n\in \mathbb Z_+$
the pairs of formal distributions
$(a\oo n b)(z), c(z)$ and $a(z), (b\oo{n} c)(z)$
are local.
\end{theorem}

For the case of formal distributions over an associative algebra $A$,
the Dong Lemma is also true (see, e.g., \cite{Roit99}).
The aim of this paper is to sudy related varieties: pre-Lie (right-symmetric),
Novikov, and pre-associative (dendriform) algebras.

The latter class of algebras is defined via two bilinear products, so
the locality statement (similar to that of Theorem~\ref{thm:DongLemma})
should be checked for eight pairs of formal distributions.
We will also mention the cases of Leibniz algebras
and associative dialgebras.

\section{Locality of formal distributions over pre-associative algebras}

The notion of a pre-associative algebra (under the term ``dendriform algebra'')
was proposed by J.-L. Loday in \cite{Loday2001}. The original definition
involves two bilinear operations
$(x,y)\mapsto x\prec y $ and $(x,y)\mapsto x\succ y$
satisfying
some identities of degree~3. These identities may be obtained
from associativity by means of a general procedure
(dendriform splitting, \cite{BCGN2013}).
Note that the same procedure transforms Lie algebra identities into pre-Lie ones.
This is why we prefer the term ``pre-associative'' rather than ``dendriform''.

It is convenient to define pre-associative algebras in terms of other
two bilinear operations.

\begin{definition}
A pre-associative algebra is a linear space $A$
equipped with two bilinear operations
  $(x,y)\mapsto x*y $
and
$(x,y)\mapsto xy$
satisfying the following axioms:
\begin{gather}
 (x*y)*z = x*(y*z), \label{eq:DendAss} \\
 (x*y)z = x(yz), \label{eq:DendL-star} \\
 x(y*z) = (xy)*z + x(yz)-(xy)z. \label{eq:DendR-star}
\end{gather}
\end{definition}

These new operations have the following relation to the original ones:
 $x*y = x\prec y + x\succ y$, and $xy = x\succ y$.

Suppose $A$ is a pre-associative algebra and
$a(z), b(z) \in A[[z,z^{-1}]]$ are two formal distributions over~$A$.
Since $A$ has two operations, the definition of locality for the pair
$(a(z),b(z))$ includes two conditions. Namely, if there exists $N\in \mathbb Z_+$
such that $ a(w)*b(z)(w-z)^N = 0$
then we say that $(a(z),b(z))$ is a $*$-local pair;
if $a(w)b(z)(w-z)^N=0$ for some $N\in \mathbb Z_+$ then
the pair $(a(z),b(z))$ is said to be  $\succ $-local.

Given $n\in \mathbb Z_+$, denote
\[
 (a\oo n b)(z) = \Res\limits_{w=0} a(w)b(z)(w-z)^n,
 \quad
 (a\soo n b)(z) = \Res\limits_{w=0} a(w)*b(z)(w-z)^n,
\]
as for ``ordinary'' algebras.

The purpose of this section is to show that the Dong Lemma does not hold
``in full'' for pre-associative algebras. First, let us prove
positive parts of the statement.

\begin{theorem}\label{thm:DongLemma_preAs}
Let $A$ be a pre-associative algebra. Suppose
$a(z),b(z),c(z) \in A[[z,z^{-1}]]$
are pairwise mutually $*$-local and $\succ $-local.
Then for every $n\in \mathbb Z_+$
\begin{itemize}
\item
the pair $a(z),(b\oo n c)(z)$ is $\succ $-local;
\item
the pair $a(z), (b\soo n c)(z)$ is $*$-local and $\succ $-local;
\item
the pair $(a\soo n b)(z), c(z)$ is $*$-local and $\succ $-local.
\end{itemize}
\end{theorem}

\begin{proof}
All these statements follow from the embedding of a pre-associative (dendriform) algebra $A$
into an associative (Rota--Baxter) algebra $\hat A$
described in \cite{GubKol2013}.
Let us recall the essential part of the construction.

As a linear space, $\hat A$ is a direct sum of two copies of the space~$A$:
\[
 \hat A = A \dotplus \bar A.
\]
A multiplication $(x,y)\mapsto x\cdot y$, for $x,y\in \hat A$, is defined by
\[
\begin{gathered}
 u\cdot v = u* v, \quad u,v\in A; \\
 \bar u\cdot v = \overline{u* v} - \overline{uv}, \quad \bar u\in \bar A,
 \ v\in A; \\
 u\cdot \bar v = \overline{uv}, \quad u\in A,\ \bar v \in \bar A; \\
 \bar u\cdot \bar v = 0,\quad \bar u,\bar v\in \bar A.
\end{gathered}
\]
Then $\hat A$ with the operation $\cdot $ is an associative algebra.
For $x\in \{a,b,c\}$ consider $x(z)$ and $\bar x(z)$ as formal distributions
over $\hat A$:
if $x(z)=\sum\limits_{n\in \mathbb Z }x_n z^{-n-1}$, $x_n\in A$,
then
\[
\bar x(z)=\sum\limits_{n\in \mathbb Z }\bar x_n z^{-n-1} , \quad \bar x_n\in \bar A\subset \hat A.
\]

By the hypothesis,
$a(z)$, $b(z)$, $c(z)$, $\bar a(z)$, $\bar b(z)$, $\bar c(z)$
are pairwise mutually local formal distributions from $\hat A[[z,z^{-1}]]$.

In particular, consider $a(z)$, $b(z)$, and $\bar c(z)$. By the
Dong Lemma for associative algebras, the pair $(a(z), \overline{(b\oo n c)}(z))$
of formal distributions over $\hat A$ is local.
Since
$(b\oo n \bar c)(z) = \overline{(b\oo n c)}(z)$ by the definition of operation in $\hat A$,
the pair $(a(z), (b\oo n c)(z))$ is $\succ $-local.

Similarly, since $\overline{(b\soo n c)}(z)=(\bar b\oo n c)(z) + (b\oo n \bar c)(z)$,
the pair $(a(z), (b\soo n c)(z))$ is $\succ $-local in $A[[z,z^{-1}]]$.
The $*$-locality of the same pair is an immediate corollary of the Dong Lemma for associative
algebras.
The remaining statements are proved in a similar way.
\end{proof}

Consider the remaining locality statements on formal distributions over
a pre-associative algebra:
$*$-locality for the pairs
$(a(z), (b\oo n c)(z))$, $((a\oo n b)(z), c(z))$,
and $\succ $-locality for $((a\oo n b)(z),c(z))$.
Let us show that these statements do not hold in general.
We will construct an appropriate pre-associative algebra~$A$
by means of generators and defining relations.

\begin{proposition}\label{prop:Dendr_nonLoc}
 There exists a pre-associative algebra $A$ and a formal distribution
 $a(z)\in A[[z,z^{-1}]]$ such that the pair $(a(z),a(z))$ is both
 $*$-local and $\succ $-local, but $(a(z), (a\oo 0 a)(z))$ is not $*$-local
 and
 $((a\oo 0 a)(z), a(z))$ is neither $*$-local, nor $\succ $-local.
\end{proposition}

\begin{proof}
Let $X=\{a_n \mid n\in \mathbb Z\}$ be a countable set indexed by integers
and let $F(X)$ stand for the free pre-associative algebra generated by~$X$
(we will recall its structure below).
Denote by $I$ the ideal of $F(X)$ generated by
the elements
\[
 a_n*a_m-a_0*a_{n+m}, \quad a_na_m - a_0a_{n+m},
\]
for all $n,m\in \mathbb Z$, $n\ne 0$. Consider the quotient
pre-associative algebra
$A=F(X)/I$.
Then the formal distribution
\[
 a(z) = \sum\limits_{n\in \mathbb Z}  (a_n+I) z^{-n-1} \in A[[z,z^{-1}]]
\]
has the following properties:
\[
 (a(w)*a(z))(w-z) = 0,\quad (a(w)a(z))(w-z) = 0.
\]
In other words, the pair $(a(z),a(z))$ is both $*$-local and $\succ $-local
with locality level $N=1$.

In order to show that the pair of formal distributions
$((a\oo 0 a)(z), a(z))$ is not $\succ $-local,
 consider the elements
\begin{equation}\label{eq:Loc2-1Dend}
f_{n,m}^{(N)} =  \sum\limits_{s=0}^N (-1)^s \binom{N}{s} (a_0a_{n-s})a_{m+s} \in F(X),
\quad n,m\in \mathbb Z,\ N\in \mathbb Z_+.
\end{equation}
It is enough to prove that for every $N\in \mathbb Z_+$
there exist $n,m \in \mathbb Z$ such that $f_{n,m}^{(N)} \notin I$.

For the latter, we may use the homogenuity of defining relations
which allows us to compute explicitly all polynomials of degree 3 from~$I$.
One may organize the computation better by making use
the Gr\"obner--Shirshov bases (GSB) method
for pre-associative (dendriform) algebras
proposed in \cite{KolComm13}, but it is not necessary
to apply this technique.

It is easy to see that the identities \eqref{eq:DendAss}--\eqref{eq:DendR-star}
allow one to rewrite an arbitrary term in the free pre-associative algebra $F(X)$
generated by a
set $X$ as a linear combination of monomials
\begin{equation}\label{eq:DendFreeBasis}
 u_1*u_2*\dots *u_k, \quad k\ge 1,
\end{equation}
where each $u_i$ is a nonassociative (magmatic) word
in $X$ involving only the second operation $(x,y)\mapsto xy$.
It was shown in \cite{KolComm13} that the monomials \eqref{eq:DendFreeBasis}
are linearly indepedent, i.e., they form a linear basis of $F(X)$.
For the set
\[
S = \{a_na_m - a_0a_{n+m},\, a_n*a_m - a_0*a_{n+m} \mid n,m\in \mathbb Z, n\ne 0 \},
\]
$S\subseteq F(X)$,
consider the space $I_3$ of all elements of degree 3 in the ideal $I$ generated by $S$:
they are spanned by $sa_k$, $a_ks$, $s*a_k$, $a_k*s$, where $s\in S$. Hence,
the space $I_3\subset I$
is spanned by
\begin{gather}
 (a_na_m)a_k - (a_0a_{n+m})a_k, \quad a_k(a_na_m) - a_k(a_0a_{n+m}), \label{eq:MM-1} \\
 (a_na_m)*a_k - (a_0a_{n+m})*a_k, \quad a_k*(a_na_m) - a_k*(a_0a_{n+m}), \label{eq:MS-1} \\
 (a_n*a_m)*a_k - (a_0*a_{n+m})*a_k, \quad a_k*(a_n*a_m) - a_k*(a_0*a_{n+m}), \label{eq:SS-1} \\
 (a_n*a_m - a_0*a_{n+m})a_k = a_n(a_ma_k) - a_0(a_{n+m}a_k), \label{eq:SM-1}\\
 a_k(a_n*a_m - a_0*a_{n+m}) = (a_ka_n)*a_m - (a_k,a_n,a_m) - (a_ka_0)*a_{n+m} + (a_k,a_0,a_{n+m}), \label{eq:MS-2}
\end{gather}
where $n,m,k\in \mathbb Z$, $n\ne 0$.
By means of linear reduction,
\eqref{eq:SM-1} and \eqref{eq:MS-2}
may be replaced with
\begin{gather}
a_n(a_0a_k) - a_0(a_{0}a_{n+k}), \label{eq:SM-1-1}\\
(a_0a_n)*a_m - (a_0a_n)a_m - (a_0a_0)*a_{n+m} + (a_0a_0)a_{n+m}, \label{eq:MS-2-1}
\end{gather}
for $n,m,k\in \mathbb Z$, $n\ne 0$.

Since all monomials of the form $(a_0a_{n-s})a_{m+s}$ that appear in $f_{n,m}^{(N)}$
are linearly independent modulo
\eqref{eq:MM-1}--\eqref{eq:SS-1}, \eqref{eq:SM-1-1}, \eqref{eq:MS-2-1},
we may conclude $f_{n,m}^{(N)}\notin I$, as desired.

In a similar way, if the pair
$((a\oo 0 a)(z), a(z))$ is  $*$-local
then there exists $N\in \mathbb Z_+$ such that
\[
g_{n,m}^{(N)} =  \sum\limits_{s\in \mathbb Z_+} (-1)^s \binom{N}{s} (a_0a_{n-s})*a_{m+s} \in I
\]
for all $n,m\in \mathbb Z$.
This is impossible since
\begin{multline*}
I\ni \sum\limits_{s\in \mathbb Z_+} (-1)^s \binom{N}{s}
  ((a_0a_{n-s})*a_{m+s} - (a_0a_{n-s})a_{m+s} - (a_0a_0)*a_{n+m}
  + (a_0a_0)a_{n+m}) )
=
 g_{n,m}^{(N)} - f_{n,m}^{(N)}
\end{multline*}
for $N>0$
(c.f. \eqref{eq:MS-2-1}), but we know $f_{n,m}^{(N)} \notin I$.

Finally, if the pair $(a(z), (a\oo 0 a)(z))$ is $*$-local then
there exists $N\in \mathbb Z_+$ such that
\[
h_{n,m}^{(N)} =  \sum\limits_{s\in \mathbb Z_+} (-1)^s \binom{N}{s} a_{n-s}*(a_0a_{m+s}) \in I
\]
for all $n,m\in \mathbb Z$.
The only option is to present $h_{n,m}^{(N)}$ as a linear combination of
elements \eqref{eq:MS-1}, but this is obviously impossible.
\end{proof}

\section{Pre-Lie algebras and the Dong Lemma for Novikov algebras}

Given a pre-associative algebra $A$ as above, its ``commutator'' algebra
$A^{(-)}$ is a right-symmetric (pre-Lie) algebra: this is the same space~$A$
equipped with bilinear operation
\[
 x\circ y  = x*y - xy -yx, \quad x,y\in A.
\]
In terms of ``original'' operations $\prec $ and $\succ$ we have
$x\circ y =x\prec y -y\succ x$.
It was proved in \cite{Gubarev} that
every right-symmetric algebra $V$ may be embedded into its universal enveloping
pre-associative algebra $U(V)$, and an analogue of the Poincar\'e--Birkhoff--Witt
Theorem holds (see also \cite{DotsTamar}).

\begin{theorem}\label{thm:DongPreLie}
Let $V$ be a right-symmetric (pre-Lie) algebra with an operation
$(x,y)\mapsto x\circ y$,
$x,y\in V$. Assume
$a(z),b(z),c(z)\in V[[z,z^{-1}]]$
are pairwise mutually local formal distributions over~$V$.
Then $((a\oo n b)(z), c(z))$ is a local pair for every $n\in \mathbb Z_+$,
but
$(a(z), (b\oo n c)(z))$ is not necessarily local.
\end{theorem}

\begin{proof}
Note that if two formal distributions over $V$ form a local pair
$(a(z),b(z))$ then we cannot say that the same series are local if
considered as formal distributions over $U(V)$. This is why
we need to recall a relation between pre-Lie and Lie algebras
\cite{GubKol2013}.

Given a right-symmetric algebra $V$,
consider the ``ordinary'' algebra $\hat V$ constructed as follows.
As a linear space, $\hat V$ is a direct sum of two copies of the space~$V$:
\[
 \hat V = V \dotplus \bar V.
\]
A multiplication $(x,y)\mapsto [xy]$, for $x,y\in \hat V$, is defined by
\[
\begin{gathered}
{}
 [uv] = u\circ v-v\circ u, \quad u,v\in V; \\
 [\bar u v] = \overline{u\circ v}, \quad \bar u\in \bar V, \ v\in V; \\
 [u\bar v] = -\overline{v\circ u}, \quad u\in V,\ \bar v \in \bar V; \\
 [\bar u\bar v] = 0,\quad \bar u,\bar v\in \bar V.
\end{gathered}
\]
Then $\hat V$ with the operation $[\cdot\cdot]$ is a Lie algebra.

Assume
\[
x(z) =\sum\limits_{n\in \mathbb Z} x_n z^{-n-1}, \quad x_n\in V,\ x\in \{a,b,c\}.
\]
Since $V\subset \hat V$, we may consider these distributions as
elements of $\hat V[[z,z^{-1}]]$.
The distributions with coefficients $\bar x_n$, $x\in \{a,b,c\}$,
belong to the same space of formal distributions over $\hat V$:
\[
\bar x(z) =\sum\limits_{n\in \mathbb Z} \bar x_n z^{-n-1}, \quad
\bar x_n\in \bar V\subset \hat V.
\]
According to the hypothesis,
the pairs of distributions
$(\bar x(z),y(z))$ are local in $\hat V[[z,z^{-1}]]$
for all $x,y\in \{a,b,c\}$.
It follows immediately from the construction that
\[
 [\bar a\oo n b](z)
 =  \sum\limits_{s\in \mathbb Z_+}
   (-1)^s \binom{n}{s} [\bar a_{n-s}b_{m+s}]
 =  \sum\limits_{s\in \mathbb Z_+}
   (-1)^s \binom{n}{s}  \overline{a_{n-s}\circ b_{m+s}}
 =  \overline{(a \oo{n} b)}(z)
\]
for every $n\in \mathbb Z_+$.
Hence, by the Dong Lemma for Lie algebras,
$\overline{(a \oo{n} b)}(z)$ and $c(z)$ form a local pair
of formal distributions over $\hat V$.
The latter means
\[
 0 = [\overline{(a \oo{n} b)}(w) c(z)](w-z)^N
 = \overline{(a\oo n b)(w)\circ c(z)} (w-z)^N
\]
for sufficiently large $N\in \mathbb Z_+$.
Since $\bar V$ is an isomorphic copy of~$V$, the first claim follows:
the pair $((a\oo n b)(z),c(z))$ is local in $V[[z,z^{-1}]]$.

In order to construct a counterexample to approve the second statement,
let us apply Gr\"obner--Shirshov bases method for right-symmetric algebras
proposed in \cite{BokChenLi}.

Suppose $X = \{a_n \mid n\in \mathbb Z \}$ is a set of generators which is well-ordered
in the following way:
\[
 a_0<a_{-1}<a_1<a_{-2}<a_2<\dots .
\]
As in \cite{BokChenLi}, denote by $\mathrm{RS}(X)$ the free right-symmetric algebra
generated by the set~$X$. Let $I$ be the ideal of $\mathrm{RS}(X)$
generated by the set
\[
S = \{h_{n,m} = a_n\circ a_m - a_0\circ a_{n+m} \mid  n,m\in \mathbb Z,\ n\ne 0\},
\]
and let $V$ stand for the quotient
$ \mathrm{RS}(X)/ I$.
Then the formal distribution
\[
 a(z) = \sum\limits_{n\in \mathbb Z} (a_n+I)z^{-n-1} \in V[[z,z^{-1}]]
\]
is local to itself with locality level $N=1$.

Note that it follows from right symmetry that
\[
(a_0\circ a_k)\circ a_m - (a_0\circ a_m)\circ a_k
 =
 a_0\circ (a_k \circ a_m) - a_0\circ (a_m \circ a_k)
\in
 a_0\circ (a_0 \circ a_{m+k}) - a_0\circ (a_0 \circ a_{k+m})
 +  I
\]
for all $n,m\in \mathbb Z$.
Hence, $(a_0\circ a_k)\circ a_m - (a_0\circ a_m)\circ a_k \in I$
for all $k, m \in \mathbb Z$.

Next, consider
\begin{multline*}
I\ni h_{1,m}\circ a_k
= a_1\circ (a_m\circ a_k) + (a_1\circ a_k)\circ a_m - a_1\circ (a_k\circ a_m) - (a_0\circ a_{m+1})\circ a_k \\
=
  a_1\circ h_{m,k} + a_1\circ (a_0\circ a_{m+k})
+ h_{1,k}\circ a_m + (a_0\circ a_{k+1})\circ a_m
- a_1 \circ h_{k,m} -a_1\circ (a_0\circ a_{k+m})
- (a_0\circ a_{m+1})\circ a_k \\
\in
    (a_0\circ a_{k+1})\circ a_m - (a_0\circ a_{k})\circ a_{m+1} +I.
\end{multline*}
Therefore,
\[
g_{m,k} =  (a_0\circ a_{k+1})\circ a_m - (a_0\circ a_{k})\circ a_{m+1} \in I
\]
for all $k,m\in \mathbb Z$.

Since the generators $S$ of $I$ are homogeneous, it is straightforward
to find all elements of degree 3 in the Gr\"obner--Shirshov bases
of~$I$.
In fact, according to the scheme from \cite{BokChenLi}
we only need to add all compositions of right multiplication
\begin{equation}\label{eq:CompPreLie}
 h_{n,m}\circ a_k, \quad a_k<a_m.
\end{equation}
It is easy to see that all these compositions are trivial modulo the set of relations
the set
\[
 S' = \{h_{n,m}, g_{m,k} \mid n,m,k \in \mathbb Z, n\ne 0 \}.
\]
Namely,
 the principal monomial of $h_{n,m}$ (according to the same order as in \cite{BokChenLi}) is $a_n\circ a_m$,
 $(a_n\circ a_m)\circ a_k$ is not a ``good'' word,
and
\begin{multline*}
h_{n,m}\circ a_k
= a_n\circ (a_m\circ a_k) + (a_n\circ a_k)\circ a_m - a_n\circ (a_k\circ a_m) - (a_0\circ a_{n+m})\circ a_k \\
=
  a_n\circ h_{m,k} + a_n\circ (a_0\circ a_{m+k})
+ h_{n,k}\circ a_m + (a_0\circ a_{n+k})\circ a_m
- a_n \circ h_{k,m} -a_n\circ (a_0\circ a_{k+m})
- (a_0\circ a_{n+m})\circ a_k \\
=
  a_n\circ h_{m,k} + h_{n,k}\circ a_m - a_n \circ h_{k,m}
  + (a_0\circ a_{n+k})\circ a_m - (a_0\circ a_{n+m})\circ a_k  \\
=
  a_n\circ h_{m,k} + h_{n,k}\circ a_m - a_n \circ h_{k,m}
  +\sum\limits_{s=0}^{k-m} g_{m+s,n+k-s}.
\end{multline*}
Principal monomials of all summands in the right-hand side are smaller than
$(a_n\circ a_m)\circ a_k$, as needed for triviality of a composition.

Hence, $S'$ is the degree 3 component of a Gr\"obner--Shirshov basis
of the ideal~$I$.

Return to the formal distribution $a(z) \in A[[z,z^{-1}]]$
and assume
$(a(z), (a\oo 0 a)(z)) $ is a local pair.
Then there exists $N\in \mathbb Z_+$ such that
\[
 \sum\limits_{s\in \mathbb Z_+} (-1)^s \binom{N}{s}
a_{n-s}\circ (a_0 \circ a_{m+s}) \in I
 \]
for all $n,m\in \mathbb Z$. Since all words in this combination are $S'$-reduced,
they are linearly indepedent modulo~$I$, a contradiction.
\end{proof}

Now we turn to the particular case when $V$ is a Novikov algebra,
i.e., satisfies the identities \eqref{eq:RSym} and \eqref{eq:LCom}.

\begin{corollary}\label{cor:NovikovDongLemma}
In the conditions of Theorem~\ref{thm:DongPreLie},
assume $V$ is a Novikov algebra.
Then the formal distributions $a(z)$, $(b\oo n c)(z)$
form a local pair in $V[[z,z^{-1}]]$.
\end{corollary}

\begin{proof}
Suppose $(a(w)\circ c(z))(w-z)^N=0$. Then for the same $N\in \mathbb Z_+$ we have
\begin{multline*}
 (a(w)\circ (b\oo n c)(z)) (w-z)^N
 =
 (a(w)\circ \Res\limits_{y=0}(b(y)\circ c(z))(y-z)^n ) (w-z)^N \\
 =
 \Res\limits_{y=0} ((a(w)\circ (b(y)\circ c(z) )) (w-z)^N ) (y-z)^n \\
 =
 \Res\limits_{y=0} (b(y)\circ (a(w)\circ  c(z) ) (w-z)^N ) (y-z)^n
 =0
\end{multline*}
due to the left commutativity \eqref{eq:LCom}.
\end{proof}

As a result, Theorem~\ref{thm:DongPreLie} and Corollary~\ref{cor:NovikovDongLemma}
imply that the Dong Lemma holds for Novikov algebras.

\begin{remark}
The locality estimates from the Dong Lemma for Lie and associative algebras provides us
the corresponding estimates for Novikov algebras.
Suppose $V$ is a Novikov algebra with a binary operation $(x,y)\mapsto x\circ y$,
and let $a(z),b(z),c(z)\in V[[z,z^{-1}]]$ be three pairwise mutually local
formal distributions over $V$:
\[
 a(w)\circ b(z)(w-z)^{N(a,b)}=a(w)\circ c(z)(w-z)^{N(a,c)} = b(w)\circ c(z)(w-z)^{N(b,c)}=0
\]
for some $N(a,b),N(a,c),N(b,c)\in \mathbb Z_+$.
Then by Corollary~\ref{cor:NovikovDongLemma}
$N(a, b\oo n c)\le N(a,c)$,
and it follows from the proof of Theorem~\ref{thm:DongPreLie}
that over the Lie algebra $\hat V$ the distributions
$\bar a(z), b(z), c(z)\in \hat V[[z,z^{-1}]]$
are also pairwise mutually local with
$N(\bar a, x)=N(a,x)$, $x\in \{b,c\}$.
Hence,
\[
 N(a\oo n b), c) = N([\bar a\oo n b], c)\le N(a,b)+N(a,c)+N(b,c)-n
\]
(see, e.g., \cite{Roit99}).
\end{remark}

\section{On the formal distributions over replicated algebras}

The class of right-symmetric algebras includes the variety of associative
algebras. On the other hand, as mentioned above,
the notion of a right-symmetric algebra
appears as a result of the dendriform splitting procedure (see \cite{BCGN2013})
applied to the defining identities of Lie algebras.
It is reasonable to consider the result of a dual procedure (called ``replication'' in \cite{BG_HK2021}) applied to the same variety of Lie algebras.
This is a well-known class of Leibniz algebras which is often considered
as a ``non-commutative'' version of Lie algebras \cite{Loday93}.

Recall that a linear space $L$ equipped with a bilinear operation $(x,y)\mapsto [xy]$
is said to be a {\em Leibniz algebra} if
\[
 [x[yz]]-[y[xz]] = [[xy]z]
\]
for all $x,y,z \in L$.

Note that for every Leibniz algebra $L$ the subspace $I\subset L$ spanned
by the set $\{[ab]+[ba] \mid a,b\in L\}$ is an ideal of $L$, and the quotient
$\bar L = L/I$ is a Lie algebra.
Define an algebra $\hat L$ in the following way.
As a linear space, $\hat L$ is a direct sum of $\bar L$ and $L$. Assuming
$\bar a = a+I\in \bar L$ for $a\in L$, define a product $(x,y)\mapsto [x, y]$
on $\hat L = \bar L\dotplus L$ by the rule
\[
 \begin{gathered}
 {}
[u,v]=0, \quad u,v\in L; \\
[\bar u,v] = [uv], \quad \bar u\in \bar L,\ v\in L; \\
[u,\bar v] = -[vu], \quad u\in L,\ \bar v\in \bar L; \\
[\bar u,\bar v] = \overline{[uv]}, \quad \bar u,\bar v\in \bar L.
 \end{gathered}
\]
Then $\hat L$ is a Lie algebra relative to the operation $[\cdot,\cdot ]$.

\begin{proposition}\label{prop:Dong-Leib}
Let $L$ be a Leibniz algebra, and let
$a(z)$, $b(z)$, $c(z)$ be pairwise local formal distributions over~$L$.
Then the pairs  of formal distributions
$(a\oo{n} b)(z), c(z)$ and $a(z), (b\oo n c)(z)$ are local
for every $n\in \mathbb Z_+$.
\end{proposition}

\begin{proof}
Since the map $\tau : L\to \bar L$, $\tau(x)=\bar x$,
is a homomorphism from the Leibniz algebra $L$ onto the Lie algebra $\bar L$,
the locality of a pair $a(z),b(z)\in L[[z,z^{-1}]]$
implies that the pairs $(\bar x(z), y(z))$, $(x(z),\bar y(z))$, $(\bar x(z), \bar y(z))$
are local in $\hat L[[z,z^{-1}]]$ for all $x,y \in \{a,b,c\}$.
By the Dong Lemma applied to the Lie algebra $\hat L$,
we obtain the pairs
$([\bar a\oo n \bar b](z), c(z))$ and $(\bar a(z), [\bar b\oo n c](z))$
are local in $\hat L[[z,z^{-1}]]$.
By the definition of $\hat L$, the latter implies the required locality.
\end{proof}

\begin{remark}
In the very similar way, the Dong Lemma holds
for formal distributions over {\em associative dialgebras} \cite{LodayPir93},
i.e.,
algebras with two bilinear operations $(x,y)\mapsto x\vdash y$ and
$(x,y)\mapsto (x\dashv y)$ satisfying the axioms
\[
 (x\dashv y)\vdash z = (x\vdash y)\vdash z, \quad
 x\dashv (y\vdash z) = x\dashv (y\dashv z)
\]
along with
replicated associativity
\[
(x\vdash y)\vdash z -  x\vdash (y\vdash z)
=
(x\vdash y)\dashv z -  x\vdash (y\dashv z)
=
(x\dashv y)\dashv z -  x\dashv (y\dashv z)=0.
\]
In this case, as for pre-associative algebras,
we have eight locality statements to be checked
similarly to what is done in Proposition~\ref{prop:Dong-Leib}.
\end{remark}

\end{document}